\documentclass{amsart}
\usepackage{color}
\usepackage{pgf}

\usepackage{tikz}

\usetikzlibrary{patterns}
\usetikzlibrary{arrows,shapes,positioning}
\usetikzlibrary{decorations.markings}
\usetikzlibrary[matrix] 

\tikzstyle directed=[postaction={decorate,decoration={markings,
    mark=at position .65 with {\arrow[arrowstyle]{latex}}}}]

\tikzstyle arrowstyle=[scale=1]

\usepackage[notref,notcite]{showkeys}
\usepackage{geometry}
\usepackage{graphicx}
\usepackage{cite}
\usepackage{amssymb}
\newtheorem{thm}{Theorem}[section]
\newtheorem{cor}{Corollary}[section]

\newtheorem{prop}{Proposition}[section]
\newtheorem{lem}{Lemma}[section]

\theoremstyle{definition}
\newtheorem{example}{Example}[section]
\newtheorem{definition}{Definition}[section]

\newcommand{\N}{\mathbb{N}}
\newcommand{\Z}{\mathbb{Z}}
\newcommand{\Rset}{\mathbb{R}}

\newcommand{\Wr}{\operatorname{Wr}}

\newcommand{\Nz}{{\N_0}}
\newcommand{\h}{H}
\renewcommand{\th}{\tilde{H}}

\newcommand{\MGH}{\operatorname{GH}}
\newcommand{\MO}{\operatorname{O}}

\newcommand{\hM}{H_M}
\newcommand{\hMp}{H_{M'}}
\newcommand{\PIV}{$\mathrm{P}_{\mathrm{IV}}\,\,$}

\begin{document}
\title{Durfee rectangles and pseudo-Wronskian equivalences for Hermite polynomials}

\author{David G\'omez-Ullate}
\address{Departamento de F\'isica Te\'orica II, Universidad Complutense de
  Madrid, 28040 Madrid, Spain.}
\address{Instituto de Ciencias Matem\'aticas (CSIC-UAM-UC3M-UCM),  C/ Nicolas Cabrera 15, 28049 Madrid, Spain.}

\author{Yves Grandati}
\address{ LCP A2MC, Universit\'{e} de Lorraine, 1 Bd Arago, 57078 Metz, Cedex 3,
  France.}
\author{Robert Milson}
\address{Department of Mathematics and Statistics, Dalhousie University,
  Halifax, NS, B3H 3J5, Canada.}
\email{david.gomez-ullate@icmat.es, grandati@univ-metz.fr,   rmilson@dal.ca}

  \title{Durfee rectangles and pseudo-Wronskian equivalences for Hermite polynomials}

\maketitle

\begin{abstract}
We study an equivalence class of iterated rational Darboux transformations applied on the harmonic oscillator, showing that many choices of state adding and state deleting transformations lead to the same transformed potential. As a by-product, we derive new identities between determinants whose entries are Hermite polynomials. These identities have a combinatorial interpretation in
  terms of Maya diagrams, partitions and Durfee rectangles, and serve
  to characterize the equivalence class of rational Darboux
  transformations. Since the determinants have different orders, we
  analyze the problem of finding the minimal order determinant in each
  equivalence class, or equivalently, the minimum number of Darboux transformations. The solution to this problem has an elegant
  graphical interpretation.  The results are applied to provide alternative and more efficient representations for exceptional Hermite polynomials and  rational solutions of the Painlev\'e IV equation. 
\end{abstract}


\section{Introduction}

In this paper we introduce an infinite number of identities between
determinants whose entries are Hermite polynomials. The remarkable
fact is that the identities involve determinants of different size.
An example of such an identity would be given by

\begin{equation}\label{ex:identity}
  \begin{vmatrix}
    \h_{1} & \h_{1}' & \h_{1}'' & \h_{1}'''\\
    \h_{2} & \h_{2}' & \h_{2}'' & \h_{2}'''\\
    \h_{3} & \h_{3}' & \h_{3}'' & \h_{3}'''\\
    \h_{6} & \h_{6}' & \h_{6}'' & \h_{6}'''
  \end{vmatrix}
  =
  \frac{1}{48} \begin{vmatrix}
    \th_{1} & \th_{1}' & \th_{1}'' \\
    \th_{2} & \th_{2}' & \th_{2}'' \\
    \th_{6} & \th_{6}' & \th_{6}'' \\
  \end{vmatrix}  
  =
  \frac{1}{7680} 
  \begin{vmatrix}
    \h_{2} & \h_{2}' \\
    \th_{3} & \th_{4} \\
  \end{vmatrix}    
\end{equation}
where $H_n(x)$ denotes the $n$-th Hermite polynomial and
$\th_{n}(x)= {\rm i}^{-n} H_n ({\rm i} x)$.  The first two
determinants are Wronskian determinants, and this type of identities
had already been shown in \cite{Felder2012}, and possibly known before
\cite{Oblomkov1999}. As we shall explain below, they are a particular
case of the whole equivalence class described in this paper that
involves a given partition and its conjugate partition.

The third determinant in \eqref{ex:identity} is not a Wronskian, but
can be constructed in a similar fashion and thus we have coined the
name \textit{pseudo-Wronskian} for it. A pseudo-Wronskian involves two
sequences, one of ordinary Hermite polynomials $\h_{n}$ and one of
conjugate Hermite polynomials $\th_{n}$. The determinant is built by
taking derivatives of the ordinary Hermite polynomials as in the usual
Wronskian, but shifting upwards the degree (which is almost an
integration) for the conjugate Hermite polynomials.

Wronskian determinants are a key object in the theory of linear
differential equations, but they also arise naturally when iterating
Darboux transformations in Schr\"odinger's equation, as it was shown
by Crum \cite{Crum}. This factorization method has found many applications \cite{deift}, ranging from the theory of integrable systems \cite{Feigin,zubelli}, soliton theory \cite{Freeman1983,Matveev} and bispectral problems \cite{grunbaum,Bakalov,grunbaum2}, among others.

Darboux transformations are also fundamentally related to exceptional
orthogonal polynomials, \cite{Gomez-Ullate2012a,Garcia-Ferrero2016},
and it was precisely this research that motivated the results reported
in this paper. The determinantal identities discussed here express an
equivalence class of rational Darboux transformations that lead to the
same transformed potential up to a spectral shift, and the two
families of entries ($\h_{n}$ and $\th_{n}$) correspond to the two
families of seed functions for state-adding and state-deleting
rational Darboux transformations in the harmonic oscillator.

At the Schr\"odinger picture, these equivalences had already been
noticed by in the Laguerre and Jacobi cases by Odake in
\cite{Odake2014}, and by the authors in \cite{krein-adler}. The most
convenient way to visualize the equivalence is by shifting the origin
in a Maya diagram, a type of diagram that was originally introduced by
Sato in integrable systems theory \cite{ohta}.  They were applied to
the context of Darboux transformations and exceptional polynomials by
Takemura \cite{Takemura2014}.

Wronskian determinants of Hermite polynomials appear already in the work of Karlin and Szeg\H{o} \cite{Karlin1960}, who provide expressions for the number of real zeros of the Wronskian of a sequence of consecutive Hermite polynomials. Those sequences for which the Wronskian determinant has no real zeros were characterized by Adler in \cite{Adler1994}, and the result for arbitrary sequences has been recently derived by Garc\'ia-Ferrero et al. \cite{Garcia-Ferrero2015}, using oscillatory type arguments. The complex zeros of these Wronskians of Hermite polynomials display very intriguing symmetric patterns on the complex plane, \cite{Felder2012,clarkson1,FP2008}.

Beyond their obvious interest in Sturm-Liouville theory, Wronskian
determinants of classical polynomials play a role in the construction
of rational solutions to nonlinear integrable equations of Painlev\'e
type. The symmetry approach based on B\"acklund transformations of
Painlev\'e equations developed by Noumi and his collaborators (see
\cite{noumi} and references therein) shows how to construct rational
solutions by applying the symmetry group to a number of seed
solutions. Maya diagrams and partitions are also an essential part of
their description. In particular, two families of rational solutions
to \PIV can be constructed using generalized Hermite and Okamoto
polynomials, which can both be expressed as Wronskians of specific
sequences of Hermite polynomials \cite{clarkson2009}.  We show how to
apply our results to provide an alternative pseudo-Wronskian
representation of Okamoto and generalized Hermite polynomials, which
happens to be more efficient in the former case.

It is also possible to view the pseudo-Wronskian determinants
introduced in this paper as an extension of Jacobi-Trudi formulas in
the theory of symmetric functions, \cite{FH}. The original
Jacobi-Trudi formulas express the Schur polynomial $s_\lambda$
associated to a given partition $\lambda$ as a determinant whose
entries are complete homogeneous symmetric polynomials, or as a
determinant whose entries are elementary symmetric polynomials
associated to the conjugate partition $\lambda'$. Jacobi-Trudi
formulas were already extended to include classical orthogonal
polynomials in \cite{Sergeev2014} and lead to exceptional orthogonal
polynomials in \cite{Grandati-Schur}. The pseudo-Wronskian
determinants in this paper can be regarded as a generalization of
Jacobi-Trudi formulas that involve not just the partition $\lambda$ or
its conjugate partition $\lambda'$, but also mixed representations
described by the Durfee symbols introduced by Andrews in
\cite{andrews3}.  Analogous results for the Laguerre and Jacobi
classes, including appropriately defined pseudo-Wronskians, will be
presented in \cite{GGUM2}.

The paper is organized as follows: in Section \ref{sec:def} we
introduce the basic definitions and we review the connection between
partitions and Maya diagrams. In Section \ref{sec:pW} we show how to
associate a pseudo-Wronskian determinant to each labelled Maya diagram
and we prove the main theorem stating the proportionality of
pseudo-Wronskian determinants for all labelled Maya diagrams in the
same equivalence class. In Section \ref{sec:min} we address the
problem of finding the minimal order pseudo-Wronskian determinant in
each equivalence class. Finally, in the last two sections we apply our
results to derive alternative and more efficient pseudo-Wronskian
representations of exceptional Hermite polynomials and special
polynomials related to rational solutions of Painlev\'e type
equations.

\section{Definitions and preliminaries}\label{sec:def}
\begin{definition}
  We define a Maya diagram to be a set of integers $M\subset\Z$ that
  contains a finite number of positive integers, and excludes a finite
  number of negative integers.  
\end{definition}
\noindent

\begin{definition}
  It is clear that if $M$ is a Maya diagram, then for $k\in \Z$ so is
  $M+k = \{ m+k \colon m\in M \}$. We say that $M$ and $M+k$ are
  \textit{equivalent Maya diagrams}, and we define an
  \textit{unlabelled Maya diagram} to be the equivalence class of Maya
  diagrams related by such shifts.
\end{definition}
We visualize a Maya diagram as a horizontally extended sequence of
filled and empty boxes, with an origin placed between the box in
position $-1$ and box in position $0$, and with filled boxes
indicating membership in $M$. By contrast, an unlabelled Maya diagram
should be visualized as a sequence of filled and unfilled boxes,
without a choice of origin.  In this formulation, the defining
assumption of a Maya diagram is that all boxes sufficiently far to the
left are filled, and that all boxes sufficiently far to the right are
empty. From a physical point of view, a Maya diagram depicts the empty and filled energy levels corresponding to the spectrum of a Hamiltonian with a pure point spectrum. At the filled energy levels the Hamitonian would have a bound state, leading to a state-deleting rational Darboux transformations while at the empty levels there would be a quasi-rational formal eigenfunction, leading to a rational state-adding transformation.\footnote{For simplicity we restrict in this paper to the case of the harmonic potential, for which only state-adding and state-deleting rational Darboux transformations exist. We postpone the analysis of hamiltonians with rational isospectral Darboux transformations to a forthcoming publication.}


\begin{definition}
  We say that a Maya diagram $M\subset \Z$ is in standard form if
  $0\notin M$ and $k\in M$ for every $k<0$.
\end{definition}
\noindent
Visually, a standard diagram has a gap just to the right of the
origin, and no gaps to the left of the origin.
\begin{prop}
  \label{prop:stdform} 
  Let $M\subset \Z$ be a Maya diagram.  Then there exists a unique
  $k\in \Z$ such that $M-k$ is in standard form.
\end{prop}
\begin{proof} The desired shift is given by $k=\min\, \Z\setminus M$.
\end{proof}

\begin{figure}
\begin{center}
  \includegraphics[width=0.9\textwidth]{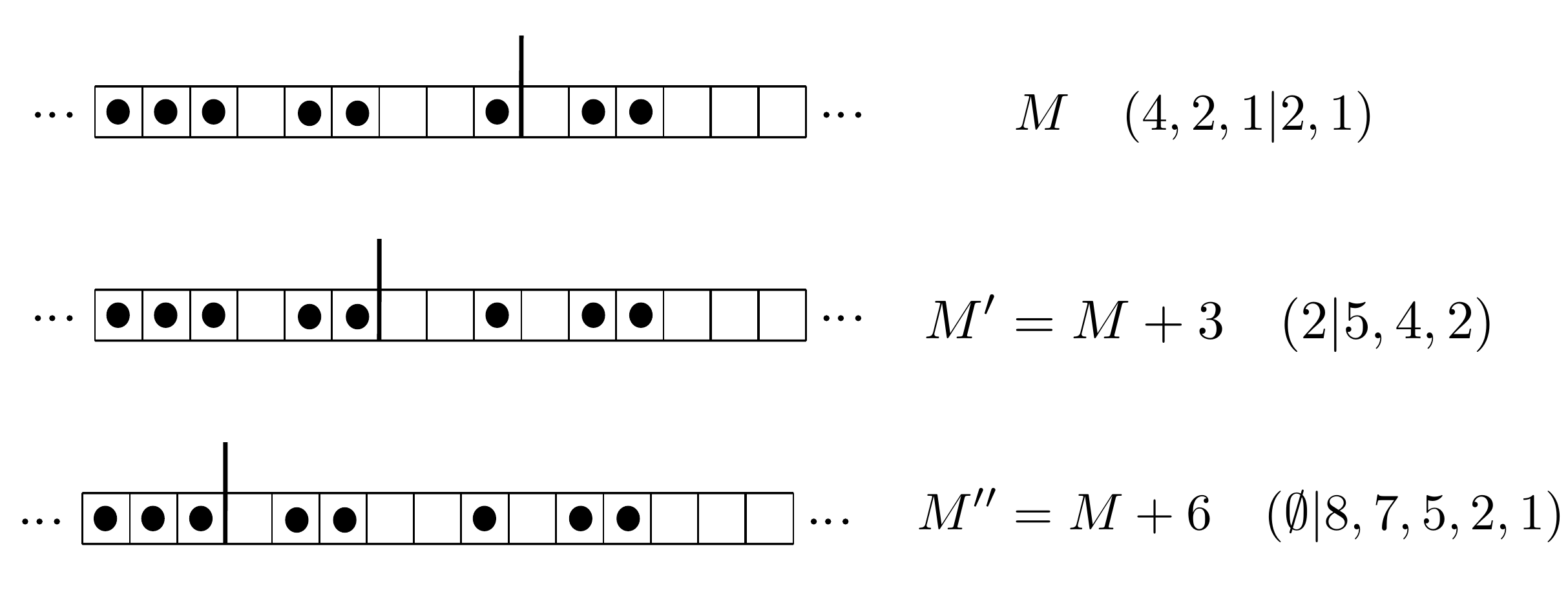}
\caption{Three equivalent Maya diagrams corresponding to the partition
$\lambda=(4,4,3,1,1)$, together with their Frobenius representation. The third diagram is in standard form.}
  \label{fig:equivM}
\end{center}
\end{figure}

\begin{definition} A partition $\lambda$ is a non-increasing sequence
  of integers $\lambda_1\geq \lambda_2\geq \cdots$ such that
  $\lambda_i = 0$ for sufficiently large $i$.  Let $\ell$ be the
  largest index such that $\lambda_\ell>0$. We call $\ell$ the length
  of the partition.  We call
  $|\lambda| = \lambda_1+\cdots+ \lambda_\ell$ the size of the
  partition.
\end{definition} 

Let $M\subset\Z$ be a Maya diagram and let $m_1>m_2>\cdots$ be its
elements ordered in decreasing order. Consider the partition defined
by
\begin{equation}
  \label{eq:lambdafromM} \lambda_i = \#\{  m\notin M \colon m < m_i\}
,\quad i=1,2\ldots .
\end{equation}
\begin{prop} The correspondence $M\mapsto \lambda$ given by
\eqref{eq:lambdafromM} defines a bijection between the set of
unlabelled Maya diagrams and the set of partitions.
\end{prop}
\begin{proof} 
  Observe that $M+k$ defines the same partition as $M$.  By Proposition
  \ref{prop:stdform} every equivalence class of Maya diagrams contains
  a unique representative in standard form.  Thus, it suffices to
  establish a 1-1 correspondence between partitions and standard
  diagrams.  The desired correspondence is given by the relation
  \begin{equation}
    \label{eq:Mfromlambda} m_i -\lambda_i =  \ell-i,\quad i=1,2,\ldots
  \end{equation} 
  If the above relation holds, then $\lambda$ is a partition of length
  $\ell$ if and only if $M$ is a standard diagram with $\ell$ positive
  elements.
\end{proof}

It is convenient to represent a partition by means of a Ferrer's
diagram, a finite collection of points arranged in left-justified
rows, with the row lengths in non-increasing order.  

\begin{definition} Let $\curlyeqprec$ denote the box partial order on
$\N\times \N$.  Formally, $(i_1, j_1) \curlyeqprec (i_2,j_2)$ if and
only if $i_1\leq i_2$ and $j_1\leq j_2$.  We now define a Ferrer's
diagram to be a finite subset of $\N\times \N$ \footnote{Throughout the paper we will use the following notation: $\N = \{1,2,\ldots \}$ and
$\Nz=\{ 0,1,2,\ldots \}$.} which is down-closed
with respect to the box order.
\end{definition}

Formally, the correspondence between Ferrer's diagrams and partitions
is as follows.  Given a Ferrer's diagram $F\subset \N\times \N$, 
let 
\[ \lambda_j = \# \{ i\in \N\colon (i,j) \in F\},\quad j=1,2,\ldots
;\]
i.e., $\lambda_j$ is the number of points in row $j$.  Conversely, if
$\lambda$ is a partition, then the corresponding Ferrer's diagram is
given by
\begin{equation}
  \label{eq:Ffromlambda} F = \{ (i,j) \in \N\times \N \colon i \leq
\lambda_j \}.
\end{equation}

\begin{definition} For a partition
$\lambda=(\lambda_1,\dots,\lambda_\ell) $ of length $\ell$, set
  \begin{equation}
    \label{eq:conjlambdadef} \lambda'_j = \# \{ i\in \N \colon
\lambda_i \geq j\},\quad j=1,2,\ldots
  \end{equation} 
\end{definition}
\noindent The resulting sequence $\lambda'$ is called the conjugate
partition of $\lambda$.  The following is well known.
\begin{prop} Let $\lambda, \lambda'$ be as above, and let $F,F'$ be
the corresponding Ferrer's diagrams.  Then $F'$ is the transpose of
$F$, meaning that
  \[ F' = \{ (j,i) \in \N\times \N \colon (i,j) \in F.\}.\]
\end{prop}

There is another way to represent Maya diagrams, one that makes the
relation to Ferrer's diagrams more explicit, and which will be useful
when we consider the minimal order problem.
\begin{definition} We define a \emph{bent Maya diagram} to be a doubly
infinite sequence $B=\{ (i_n,j_n)\in \Nz\times \Nz\colon n\in \Z\}$
such that 
\[ (i_{n+1},j_{n+1})-(i_n, j_n ) \in \{ (1,0), (0,-1) \} ,\quad n \in
\Z\] and such that $i_n j_n= 0$ for all but finitely many $n$.
\end{definition}
\noindent
Note that since the displacement $(i_n,j_n)\mapsto (i_{n+1},j_{n+1})$
is either down or to the right, the above definition implies that
$i_n=0$ for all $n$ sufficiently small and that $j_n=0$ for all $n$
sufficiently large.

\begin{prop}
  For a Maya diagram   $M\subset \Z$ set
  \begin{equation}
    \label{eq:injndef} 
    i_n = \# \{ m\notin M : m< n \},\qquad j_n = \# \{
    m\in M : m\geq n \},\quad n\in \Z.
  \end{equation} 
  Then, the doubly infinite sequence $B=\{(i_n,j_n)\}_{n\in \Z}$ is a
  bent Maya diagram.
\end{prop}
\begin{proof}
  By assumption, there exists an $N>0$ such that $n\notin M$ for all
  $n\geq N$ and such that $n\in M$ for all $n\leq-N$.  Hence, $j_n=0$
  for all $n\geq N$ and $i_n=0$ for all $n\leq -N$.  If $n\in M$, then
  $(i_{n+1},j_{n+1})=(i_n,j_n-1)$.  If $n\notin M$, then
  $(i_{n+1},j_{n+1})=(i_n+1,j_{n+1})$.  Therefore, in both cases the
  defining condition of a bent diagram is satisfied.
\end{proof}
\noindent Informally, a bent Maya diagram is a 2-dimensional
representation of a Maya diagram, with a filled box at position $n$
corresponding to a unit downward displacement $(0,-1)$ and an empty
box corresponding to a unit rightward displacement $(1,0)$, as depicted in Figure \ref{fig:maya-partition}. A
translation $M'= M-k$ corresponds to an index shift in the
bent diagram: 
\begin{equation}
  \label{eq:injnshift}
   (i'_n,j'_n) = (i_{n+k},j_{n+k}).
\end{equation}


There is a connection between bent diagrams and Ferrer's diagrams.
\begin{definition} Let $F\subset \N\times \N$ be a Ferrer's diagram.
We define the \emph{rim} of $F$ to be subset
\[F'=\{ (i,j)\in F \colon (i+1,j+1)\notin F\}.\]
\end{definition}

\begin{prop} 
  \label{prop:bentmaya}
  Let $M\subset \Z$ be a Maya diagram, $B$ the corresponding bent
  diagram defined by \eqref{eq:injndef}, $\lambda$ the corresponding
  partition defined by \eqref{eq:lambdafromM}, and $F$ the
  corresponding Ferrer's diagram defined by
  \eqref{eq:Ffromlambda}. Then, $F' = B\cap (\N\times \N)$.
\end{prop}
\noindent
Thus, the rim is that finite subset of the bent diagram whose points
have non-zero coordinates.

\begin{proof} 
  We first show that the non-zero part of $B$ lies in $F'$.
  Suppose that $i_n,j_n>0$. Let $m_1>m_2>\cdots$ be the elements of
  $M$ in decreasing order.  By \eqref{eq:injndef},
  $m_{j_n+1}< n \leq m_{j_n}$.  Since
  \[ \{ m\notin M \colon m< m_{j_n+1}\}\subset \{ m\notin M \colon
  m<n\}\subset \{ m\notin M \colon m< m_{j_n}\}\]
  it follows that $\lambda_{j_n+1}\leq i_n \leq \lambda_{j_n}$.
  Therefore by \eqref{eq:Ffromlambda}, $(i_n,j_n)\in F$ but
  $(i_{n}+1,j_{n}+1)\notin F$.

  We now prove the converse.  Arguing by contradiction, suppose that
  there is an $(i,j)\in F'$ which does not belong to $B$.  In
  $F'\setminus B$ choose the points with $j$ as small as possible, and
  of those choose the point that has $i$ as large as possible. By
  assumption, $i\leq \lambda_j$.  We consider two cases.

  Case 1: assume that $i<\lambda_j$.  Then $i+1\leq \lambda_j$, which
  means that $(i+1,j)\in F$.  Since $(i+1,j+1) \notin F$, the same is
  true for $(i+2,j+1)$.  Hence, $(i+1,j)\in F'$ also.  Because of the
  assumed maximality of $i$, we must have $(i+1,j)\in B$;
  i.e. $(i+1,j)=(i_n,j_n)$ for some $n\in \Z$. By the definition of a
  bent diagram, $(i_{n-1},j_{n-1})$ is either $(i_n,j_n+1)=(i+1,j+1)$
  or $(i_n-1,j_n)= (i,j)$.  The second possibility is excluded because
  we have assumed that $(i,j)\notin B$.  By the first part of the
  proof, $(i_{n-1},j_{n-1})\in F'$.  This means that $(i+1,j+1)\in F$,
  which contradicts the assumption that $(i,j)\in F'$.

  Case 2: assume that $i=\lambda_j$.  If $j>1$, then
  $i\leq \lambda_{j-1}$ which means that $(i,j-1)\in F$.  By
  assumption, $(i+1,j)\notin F$.  Hence $(i,j-1)\in F'$ and hence
  $(i,j-1)\in B$ by the minimality of $j$.  We now repeat the above
  argument to conclude that $(i,j)\in B$ also --- a
  contradiction. Hence $j=1$ and $i=\lambda_1$.  Set $n=\max M$.  By
  \eqref{eq:lambdafromM}, $i_n=\lambda_1$ and $j_n=1$. This
  contradicts the assumption that $(i,j)\notin B$.
\end{proof}
\begin{definition}
For a given Maya diagram $M\subset \Z$, define
\[ M_+ = \{ m\in M \colon m\geq 0\},\qquad M_-=\{ -m-1 \colon m<0,
m\notin M\}.\]
\end{definition}
In other words, $M_+$ gives positions of the filled boxes to the right
of the origin, and $M_-$ the positions of the holes to the left of the
origin.  The numbers in $M_+$ and $M_-$ indicate distance to the
origin, with $0$ indicating a position adjacent to the origin. Since
$M_+$ and $M_-$ fully define $M$, the defining assumptions of a Maya
diagram are equivalent to the condition that $M_+$ and $M_-$ should be
finite subsets of $\Nz$.

\begin{definition}
  Let $M\subset \Z$ be a Maya diagram. Let $\{s_1,\dots,s_p\},\,  p=i_0,$
  be the elements of $M_-$ and $\{t_1,\dots,t_q\},\, q=j_0,$ the
  elements of $M_+$, arranged in decreasing order.  The double list
   $(s_1,\ldots, s_p \mid t_1,\ldots, t_q)$ is called the
 \textit{Frobenius symbol of $M$ }\cite{olsson}.
\end{definition}

The classical Frobenius symbol \cite{andrews,olsson,andrews2,andrews3}
corresponds to the case where $i_0=j_0$; i.e. the case where $M_-$ and
$M_+$ have the same cardinality.  Such a choice of origin can be
visualized as the unique intersection of the rim and the main diagonal
in $\N\times \N$.

\begin{figure}
\begin{center}
\includegraphics[width=0.4\textwidth]{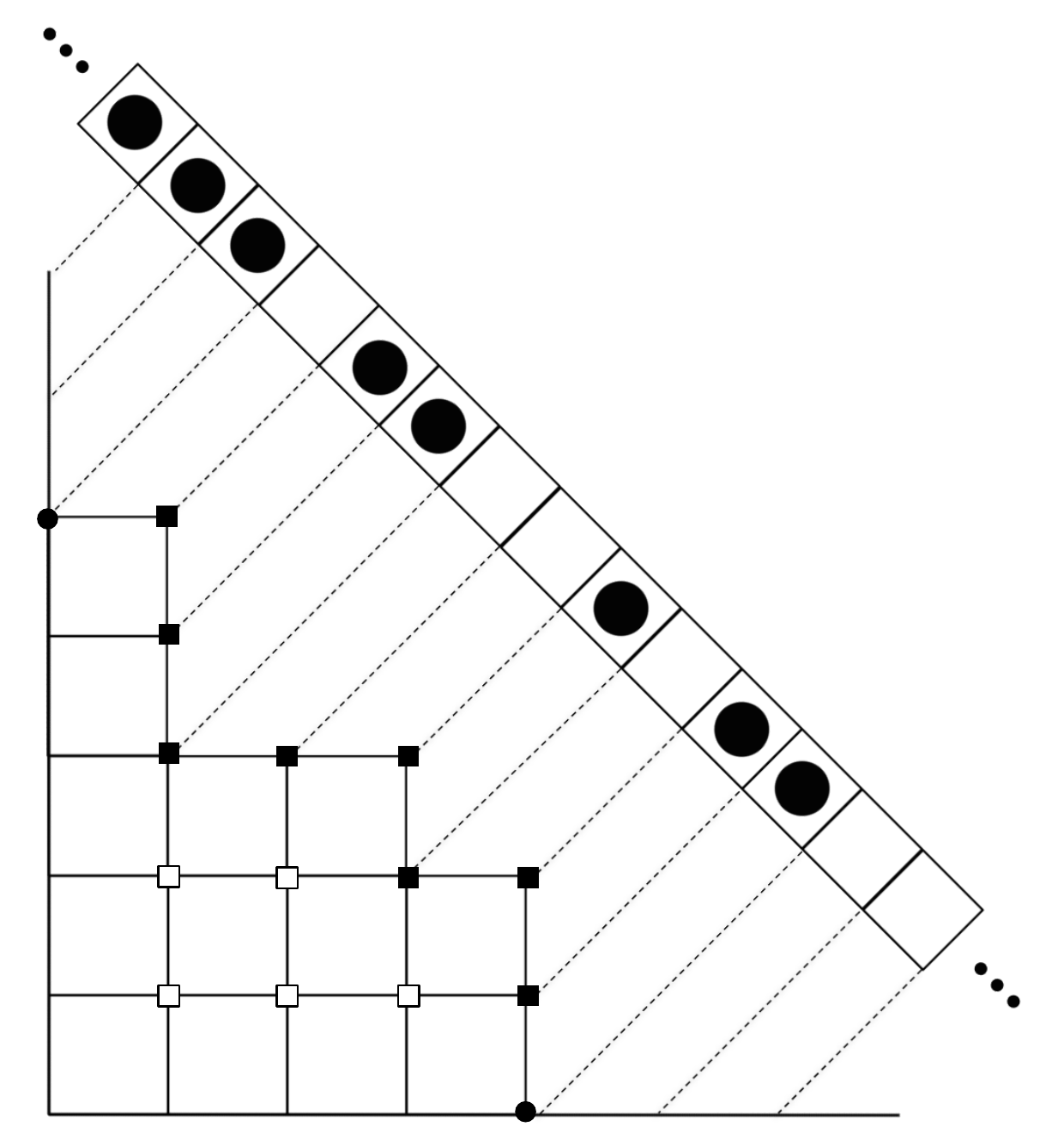}
\caption{Correspondence between an unlabelled Maya diagram and a
partition $\lambda=(4,4,3,1,1)$. The squares are points that belong to
the Ferrer diagram $F$, while black squares belong to the rim of $F$.}
\label{fig:maya-partition}
\end{center}
\end{figure}

Let us also note the following connection between the Frobenius symbol and bent diagrams. 
\begin{prop}
  Let $M\subset \Z$ be a Maya diagram and  $B=\{ (i_n,j_n)\}_{n\in
    \Z}$ the corresponding bent diagram. Then, $i_n$ and $j_n$ are the
  cardinalities of $(M-n)_-$ 
and $(M-n)_+$, respectively.
\end{prop}
\begin{proof}
  By \eqref{eq:injndef}, $i_0$ is the cardinality of $M_-$ and $j_0$
  the cardinality of $M_+$.  The general relation follows by
  \eqref{eq:injnshift}.
\end{proof}

\section{Hermite pseudo-Wronskians}\label{sec:pW}
In this section we will associate to each labelled Maya diagram a
certain determinant whose entries are Hermite polynomials. We then prove
that determinants associated to equivalent Maya
diagrams are proportional to each other. This is our main result in
this Section.  

For $n\geq 0$, let
\begin{equation}
  \label{eq:Hndef}
  \h_n(x) = (-1)^n e^{x^2} D_x^n e^{-x^2},\quad D_x=\frac{d}{dx}
\end{equation}
denote the degree $n$ Hermite polynomial, and
\begin{equation}
  \label{eq:thndef}
  \th_n(x)={\rm i}^{-n} \h_{n}({\rm i}x)
\end{equation}
the conjugate Hermite polynomial.  Recall that $y=\h_n$ is a
solution of the Hermite differential equation
\begin{equation}
  \label{eq:hermiteDE}
  y''-2x y' + 2n y=0
\end{equation}
for $n\geq 0$,
and that $y=e^{x^2} \th_{-n-1}$ is a solution of \eqref{eq:hermiteDE}
for $n<0$.

\begin{definition}\label{def:HpW}
  Let $M\subset \Z$ be a Maya diagram. Let $\{s_1,\dots,s_p\}$ be the
  elements of $M_-$ and $\{t_1,\dots,t_q\}$ the elements of $M_+$, both
  arranged in descending order. We define the \textit{Hermite
    pseudo-Wronskian} associated to $M$ to be
  \begin{equation}\label{eq:pWdef1}
    \hM = e^{-px^2}\Wr[ e^{x^2} \th_{s_1},\ldots, e^{x^2}
    \th_{s_p}, \h_{t_q},\ldots \h_{t_1} ],
  \end{equation}
  where $\Wr$ denotes the Wronskian determinant of the indicated
  functions.
\end{definition}
\noindent The polynomial nature of $\hM$ becomes evident once we
represent it using a slightly different determinant.

\begin{prop}
  A Hermite pseudo-Wronskian admits the following alternative
  determinantal representation
  \begin{equation}\label{eq:pWdef2}
    \hM =
    \begin{vmatrix}
      \th_{s_1} & \th_{s_1+1} & \ldots & \th_{s_1+p+q-1}\\
      \vdots & \vdots & \ddots & \vdots\\
      \th_{s_p} & \th_{s_p+1} & \ldots & \th_{s_p+p+q-1}\\
      \h_{t_q} & D_x \h_{t_q} & \ldots & D_x^{p+q-1}\h_{t_q}\\
      \vdots & \vdots & \ddots & \vdots\\
      \h_{t_1} & D_x \h_{t_1} & \ldots & D_x^{p+q-1}\h_{t_1}
    \end{vmatrix}
  \end{equation}
\end{prop}
\begin{proof}
  The desired conclusion follows by the fundamental identities
  \begin{equation}
    \label{eq:hermids}
    \begin{aligned}    \noindent
      &D_x \h_n(x) = 2n \h_{n-1}(x),\quad n\geq 0,\\     \noindent
      &D_x \th_n(x) = 2n \th_{n-1}(x),\quad n\geq 0,\\     \noindent
      &2x \h_n(x) = \h_{n+1}(x) + 2n \h_{n-1}(x),\\     \noindent
      &2x \th_n(x) = \th_{n+1}(x) - 2n \th_{n-1}(x),\\
      &D_x (e^{x^2} \th_n(x)) 
        = e^{x^2}\th_{n+1}(x),\\
      &D_x (e^{-x^2} h_n(x)) = -e^{-x^2} h_{n+1}(x).
    \end{aligned}
  \end{equation}
  and the Wronskian identity
  \begin{equation}
    \label{eq:Wrhomog}
    \Wr[g f_1,\ldots, g f_s] = g^s \Wr[f_1,\ldots, f_s], 
  \end{equation}
\end{proof}

We refer to $\hM$ as a pseudo-Wronskian because it is constructed by
means of a modified Wronskian operator that replaces $D_x$ with an
indicial shift for the rows with the conjugate Hermites. More
specifically, if $j\in M_+$, then the $(j,k)$ entry is
$D_x^{(k-1)}H_j$, as in the ordinary Wronskian. If $i\in M_-$, then
the $(i,k)$ entry is the conjugate Hermite polynomial $\th_{i+k-1}$.

\begin{example}\label{ex:pw}
  Consider the first Maya diagram in Figure \ref{fig:equivM}.  The
  Frobenius symbol is $(5,2,1\mid 2,1)$.  The Hermite pseudo-Wronskian
  $H_M$ associated to $M$ is given by
  \[
  \hM= {\rm e}^{-3x^2} \Wr[{\rm e}^{x^2} \th_{5},{\rm e}^{x^2} \th_{2},{\rm e}^{x^2} \th_{1},H_1,H_2 ]=  \begin{vmatrix}
    \th_{5} & \th_{6} & \th_{7} &\th_{8} & \th_{9}\\
    \th_{2} & \th_{3} & \th_{4} &\th_{5} & \th_{6}\\
    \th_{1} & \th_{2} & \th_{3} &\th_{4} & \th_{5}\\
    \h_{1} & \h_{1}' & \h_{1}'' & \h_{1}''' & \h_{1}^{(4)}\\
    \h_{2} &  \h_{2}' & \h_{2}'' & \h_{2}''' & \h_{2}^{(4)}
  \end{vmatrix}
  \]

\end{example}

The main result of this section is the following class of fundamental
determinantal identities enjoyed by these polynomials
\begin{thm}\label{thm:detequiv}
  Let $M$ and $M'=M-k,\; k>0$ be two equivalent Maya diagrams.  Set
  \begin{align*}
    E_k&= \{ m\in M \colon 0\leq m < k \} 
    & G_k&= \{ m\notin M \colon 0\leq m < k \}\\
    \epsilon_i &= (-1)^{\# \{ m\notin M \colon m<i \}} \prod_{m\in M\atop m>i} (2m-2i)   
    &  \gamma_i &= (-1)^{\# \{ m\in M \colon m>i \}} \prod_{m\notin M\atop  m<i} (2m-2i)
  \end{align*}
  Then,
  \begin{equation}
    \label{eq:detequiv}
     \left(\prod_{i\in G_k} \gamma_i\right) H_{M'} = \left(\prod_{i\in
      E_k} \epsilon_i \right) H_M
  \end{equation}
\end{thm}

\begin{proof}
  Once the following two Lemmas are established, it suffices to verify
  that the factors shown in \eqref{eq:shift1} and \eqref{eq:shift2}
  below are equal to the above-defined $\gamma_i$ and $\epsilon_i$
  symbols, respectively.
\end{proof}
Throughout, $\{s_1,\dots,s_p\} $ and $\{t_1,\dots,t_q\}$
are, respectively, the elements of $M_-$ and $M_+$ arranged in
descending order.
\begin{lem}
  \label{lem:shift1}
  Suppose that $M'=M-1$ and that $0\in M$.   Then,
  \begin{equation}
    \label{eq:shift1}
    H_M = (-1)^{p}\,2^{q-1}\left( \prod_{b=1}^{q-1} t_b\right) H_{M'}.
  \end{equation}
\end{lem}
\begin{proof}
  By assumption, $t_q=0$ and
  \[ M'_- = \{s_1+1,\ldots, s_p+1\},\quad M'_+ = \{ t_1-1,\ldots,
  t_{q-1}-1\}.\]
  The identity
  \begin{equation}
    \label{eq:Wr1}
    \Wr[1, f_1,\ldots, f_s] = \Wr[Df_1,\ldots, Df_s],
  \end{equation}
  together with \eqref{eq:hermids} implies that
  \begin{align*}
    \hM &=  e^{-p x^2} \Wr[ 
          e^{x^2} \th_{s_1},\ldots, e^{x^2}
          \th_{s_p},1,H_{t_{q-1}},\ldots , \h_{t_q}],\\
        &=(-1)^p e^{-px^2} \Wr[  D(e^{x^2}
          \th_{s_1}), \ldots,  D(e^{x^2} \th_{s_p}), D
          \h_{t_{q-1}},\ldots, D \h_{t_1}]\\
        &=(-1)^{p} 2^{q} \left(\prod_{b=1}^{q-1} t_b\right) e^{-p x^2}
          \Wr[  e^{x^2}
          \th_{s_1+1}, \ldots,  e^{x^2} \th_{s_p+1}, 
          \h_{t_{q-1}-1},\ldots,  \h_{t_1-1}]
  \end{align*}
\end{proof}

\begin{lem}
  \label{lem:shift2}
  Suppose that $M'=M+1$ and that $-1\notin M$.  Then,
  \begin{equation}
    \label{eq:shift2}
    H_M =  (-1)^{p+q-1}2^{p-1}\left( \prod_{a=1}^{p-1} s_a\right) \hMp
  \end{equation}
\end{lem}
\begin{proof}
  By assumption, $s_p = 0$ and
  \[ M'_-=\{s_1-1,\ldots, s_{p-1}-1\},\quad M'_+=\{t_1+1,\dots,
  t_q+1\}.\] The identity \eqref{eq:Wrhomog}
  together with   \eqref{eq:Wr1} implies
  \begin{align*}
    \hM &=  e^{-p x^2} \Wr[e^{x^2}
          \th_{s_1},\ldots,e^{x^2} \th_{s_{p-1}},
          e^{x^2},h_{t_q},\ldots  \h_{t_1}]\\ 
        &=  e^{qx^2}\Wr[\th_{s_1},\ldots,\th_{s_{p-1}},
          1,e^{-x^2} h_{t_q},\ldots , e^{-x^2}\h_{t_1}]\\
        &= (-1)^{p-1} e^{q x^2} \Wr[D\th_{s_1},\ldots,D\th_{s_{p-1}},
          De^{-x^2} h_{t_q},\ldots , De^{-x^2}\h_{t_1}]\\ 
        &= (-1)^{p+q-1} 2^{p-1} \left(\prod_{a=1}^{p-1} s_a\right) e^{q x^2}
          \Wr[\th_{s_1-1},\ldots,\th_{s_{p-1}-1},
          e^{-x^2} h_{t_q+1},\ldots , e^{-x^2}\h_{t_1+1}]\\ 
        &= (-1)^{p+q-1} 2^{p-1} \left(\prod_{a=1}^{p-1} s_a\right)
          e^{-(p-1) x^2}
          \Wr[e^{x^2} \th_{s_1-1},\ldots,e^{x^2}\th_{s_{p-1}-1},
          H_{t_q+1},\ldots , H_{t_1+1}]
  \end{align*}
\end{proof}

Proposition \ref{prop:stdform} immediately gives the following.
\begin{cor}
  \label{cor:uniquewronsk}
  Every Hermite pseudo-Wronskian is a scalar multiple of a Hermite
  Wronskian
  \[ \Wr[H_{m_\ell}, \ldots, H_{m_1}] \]
  for some unique choice of positive integers $m_1>\cdots > m_\ell>0$.
\end{cor}
\begin{example}
  Consider the three equivalent Maya diagrams in Figure \ref{fig:equivM}.  We have
  \begin{eqnarray*}
   M'=M+3 :\quad &M'_-= \{2\}\quad &M'_+=\{2,4,5\}\\
   M''=M+6:\quad &M''_-=\emptyset \quad &M''_+=\{1,2,5,7,8\}
   \end{eqnarray*}
 Hence,
  \[ 
  H_{M'}=  \begin{vmatrix}
    \th_{2} & \th_{3} & \th_{4} &\th_{5}\\
    \h_{2} & \h_{2}' & \h_{2}'' & \h_{2}'''\\
    \h_{4} & \h_{4}' & \h_{4}'' & \h_{4}'''\\
    \h_{5} & \h_{5}' & \h_{5}'' & \h_{5}'''
  \end{vmatrix}, 
  \]
   \[H_{M''}=\Wr[H_1,H_2,H_5,H_7,H_8]\]
   By \eqref{eq:detequiv}, 
  \[ -483840 H_{M''}=-1935360 H_{M'}= H_{M} .\]
\end{example}

Before continuing, we mention that a ``pure'' pseudo-Wronskian
corresponding to a Maya diagram without any positive elements can also
be expressed as a Wronskian determinant of conjugate Hermite
polynomials.
\begin{prop}
  Let $M\subset \Z$ be a Maya diagram consisting entirely of negative
  integers, that is $M_+=\emptyset$.  Let $\{ s_1, \ldots, s_p \}$ be
  the elements of $M_-$ arranged in descending order.   Then,
  \[ H_M = \Wr[\th_{s_1},\ldots, \th_{s_p}].\]  
\end{prop}
\begin{proof}
  Recall that
  \[ H_M = e^{-px^2} \Wr[e^{x^2} \th_{s_1},\ldots, e^{x^2} \th_{s_p}]\]
  The desired identity now follows using \eqref{eq:Wrhomog}.
\end{proof}
\noindent
As a special case of Theorem \ref{thm:detequiv} we obtain the
following identity; see \cite{Oblomkov1999} \cite{Felder2012} and the
references therein.
\begin{cor}
  Let $\lambda$ be a partition of length $\ell$, $\lambda'$ the
  conjugate partitions of length $\ell'=\lambda_1$. Let
  $m_1>\ldots> m_\ell$ the positive integers defined by
  \eqref{eq:Mfromlambda} and $m'_1> \ldots > m'_{\ell'}$ the analogous
  integers for $\lambda'$. Then,
  \[ 2^{\ell'} V(m'_1,\ldots, m'_{\ell'}) \Wr[H_{m_1},\ldots,
  H_{m_\ell}] = 2^\ell V(m_1,\ldots, m_\ell) \Wr[
  \th_{m'_{1}},\ldots, \th_{m'_{\ell'}}],\] where
  \[ V(a_1,\ldots, a_k) = \prod_{1\leq i<j<k} (a_j-a_i) \]
  is the usual Vandermonde determinant.
\end{cor}

\section{The minimal order problem}\label{sec:min}
It is quite remarkable to have an infinite family of identities among
determinants of different order.  In this section we pose and solve
the following question: given an unlabelled Maya diagram, which of the
corresponding equivalent pseudo-Wronskian determinants has the
smallest order?  This question will have some applications in the
following sections to derive simpler, alternative representations of
certain class of special functions.  The precise formulation of the
question requires the following.
\begin{definition}
  We define $|M|$, the \emph{girth} of a Maya diagram $M$, to be the
  length of its Frobenius symbol, that is the sum of the cardinalities
  of $M_-$ and $M_+$.  Since the girth of $M$ is just the order of the
  corresponding pseudo-Wronskian determinant $\hM$, our aim is to
  determine
  \[ \min \{ |M-k| \colon k \in \Z\},\]
  a quantity that we will call the \emph{minimal girth} of $M$.  A
  $k\in \Z$ such that $|M-k|$ is minimal will be called a minimal girth origin for $M$. 
\end{definition}
\noindent
As we now show, it suffices to check for minimal girth at a finite
number of $k$-values.

\begin{prop}
  \label{prop:everymooissacred}
  If $k\in \Z$ is a minimal girth origin for a Maya diagram $M$, then necessarily $k-1\in M$ and $k\notin M$.
\end{prop}
\noindent
More plainly, the minimal girth origins of a Maya diagram must occur at locations
where a full box is succeeded by an empty box.
\begin{proof}
  Let $(s_1,\ldots,s_p\mid t_1,\ldots, t_q)$ be the Frobenius symbol
  of $M-k$.  If $k\in M$, then $t_q=0$.  Observe that the Frobenius
  symbol of $M'=M-k-1$ is
  $(s_1+1,\ldots , s_p+1\mid t_1-1 , \cdots , t_{q-1}-1)$.  If
  $k-1\notin M$, then $s_p=0$. In this case, observe that the
  Frobenius symbol of $M'=M-k+1$ is
  $(s_1-1,\ldots, s_{p-1}-1\mid t_1+1,\ldots, t_q+1)$. In both cases,
  $|M'| = |M-k|-1$.  Therefore $k\notin M$ and $k-1\in M$.
\end{proof}


The solution of the minimal order problem is closely related to the
geometry of Ferrer's diagrams.
\begin{definition}
  Let $F\subset \N\times \N$ be a Ferrer's diagram corresponding to a
  partition $\lambda$.  We call $(i,j)\in F$ an \emph{inside corner}
  if $(i,j+1), (i+1,j)\in F$ but $(i+1,j+1)\notin F$.
\end{definition}

\begin{prop}
  Let $B=\{(i_n,j_n)\}_{n\in \Z}$ be the bent diagram corresponding to
  a Maya diagram $M\subset \Z$.  Then $|M-n| = i_n+j_n.$
\end{prop}
\begin{proof}
  Recall that $i_n$ is the cardinality of $(M-n)_-$ while $j_n$ is the
  cardinality of $(M-n)_+$.
\end{proof}

\begin{prop}
  \label{prop:minMGO}
  Let $n\in Z$ be a minimal girth origin for a Maya diagram $M\subset \Z$ corresponding
  to a partition $\lambda =(\lambda_1,\ldots, \lambda_\ell)$.    Then, one of
  the following three possibilities holds:
  \begin{itemize}
  \item[(a)] $(i_n,j_n)$ is an
    inside corner of the corresponding Ferrer's diagram,
  \item[(b)] $(i_n,j_n)=(0,\ell)$ 
  \item[(c)] $(i_n,j_n) = (\lambda_1,0)$.
  \end{itemize}
\end{prop}
\begin{proof}
  Without loss of generality, assume that $M$ is in standard form.  By
  Proposition \ref{prop:everymooissacred}, $n-1\in M$ and $n\notin M$.
  Hence, by \eqref{eq:injndef}, $(i_{n-1},j_{n-1}) = (i_n,j_n-1)$ and
  $(i_{n+1},j_{n+1}) = (i_n+1,j_n)$.  If $0<n<\max M$, then
  $i_n,j_n>0$ and hence by Proposition \ref{prop:bentmaya},
  $(i_n,j_n)$ is an inside corner of $F$. The only other possibilities
  are $n=0$ and $n=m_1+1$, where $m_1=\max M$.  In the former case,
  $(i_0,j_0)=(0,\ell)$. In the second case,
  $(i_{m_1+1},j_{m_1+1}) = (\lambda_1,0)$.
\end{proof}

\begin{cor}
  \label{cor:MGO}
  Let $M\subset\Z$ be a Maya diagram and $\lambda$ the corresponding
  partition.  Then, the minimal girth is given by
  $\min \{ \lambda_{j+1} + j \colon 0\leq j\leq \ell \}$.
\end{cor}
\begin{proof}
  By \eqref{eq:Ffromlambda} the inside corners of the corresponding
  Ferrer's diagram occur at $(\lambda_{j+1},j)$ for
  $j=1,\ldots, \ell-1$.  If $j=0$, then
  $(\lambda_{j+1},j) = (\lambda_1,0)$.  If $j=\ell$, then
  $(\lambda_{j+1},j) = (0,\ell)$. The desired conclusion now follows
  by Proposition \ref{prop:minMGO}
\end{proof}

\begin{definition}
  Let $F$ be a Ferrer's diagram.  We define a Durfee rectangle
  \cite{andrews3} of $F$ to be a rectangle with vertices
  $(0,0), (i,0), (i,j), (0,j)$ where $(i,j)$ is an inside corner of
  $F$.  If $M$ is the corresponding Maya diagram with an origin
  located at an inside corner, then the corresponding Frobenius symbol
  $(s_1,\ldots, s_p \mid t_1, \ldots, t_q)$ satisfies $s_p,t_q>0$.  We
  are thus able to define the partitions
  \begin{align*}
    \mu &= (s_1-p+1, \ldots, s_{p-1}+1, s_p) \\
    \nu &= (t_q-q+1,\ldots, t_{q-1}+1, t_q)
  \end{align*}
  of length $p$ and $q$, respectively.  We call
  $[\mu\mid \nu]_{p\times q}$ the Durfee symbol of $M$.
\end{definition}

Visually, the partitions $\mu$ and $\nu$ describe the complement of
the Durfee rectangle in $F$ (see Figure \ref{fig:durfee}).  Partition
$\mu$ is the transpose of the remnant above the rectangle, and $\nu$
is the remnant to the right of the rectangle.  The girth of the
corresponding Maya diagram is simply the distance of the inside corner
to the origin relative to the taxi-cab metric.  To determine the
minimal girth, the corresponding girths have to be compared to the
height $\ell$ of the Ferrer's diagram, and the width $\lambda_1$,
which can be considered as Durfee rectangles of width zero and height
zero, respectively (black circle dots in Figure \ref{fig:durfee}).
\begin{figure}\label{fig:durfee}
\begin{center}
\includegraphics[width=0.9\textwidth]{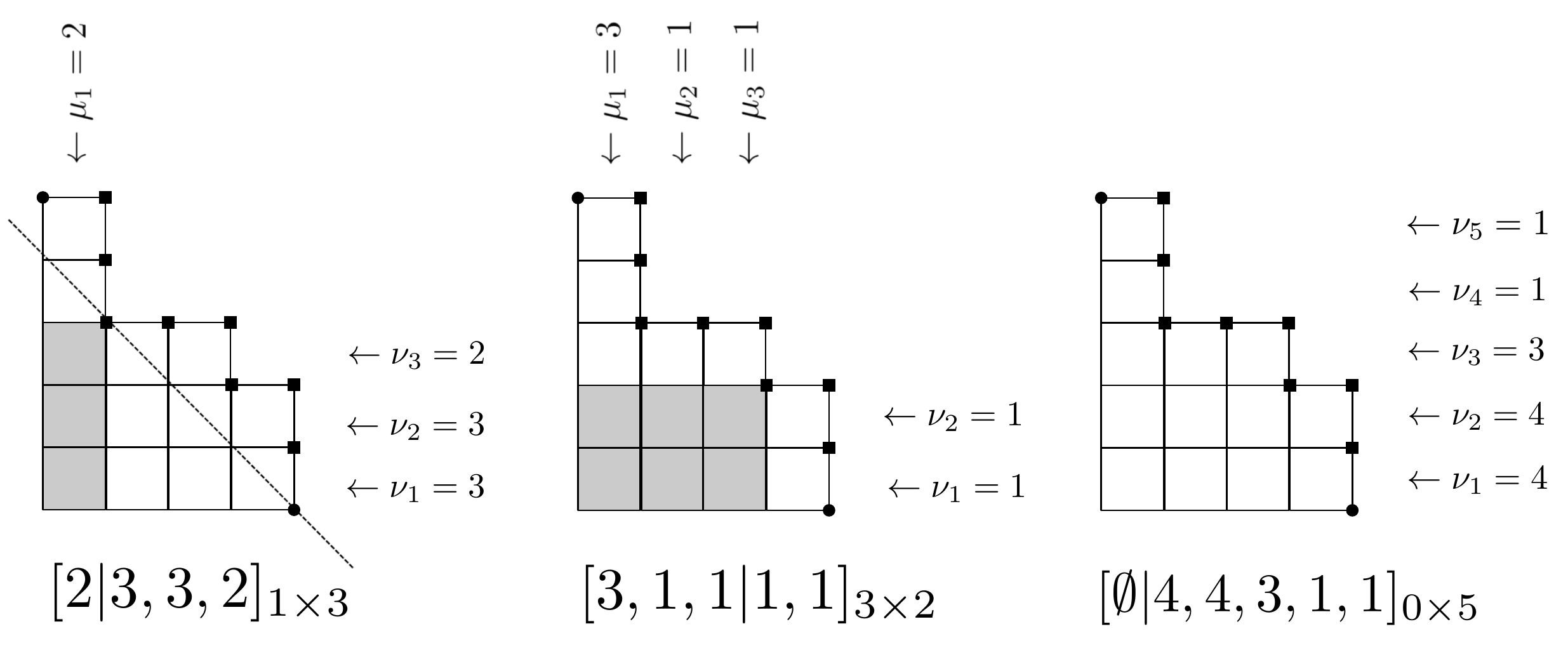}
\caption{Durfee rectangles and Durfee symbols for the three
  equivalent Maya diagrams depicted in Figure \ref{fig:equivM}. Note
  that the shortest girth corresponds to the left diagram, which is a
  non-unique solution to the minimal order problem for this
  partition. The right diagram where the Durfee symbol
  coincides with the partition, corresponds to the Maya diagram in
  standard form.}
\end{center}
\end{figure}


\section{Application to Exceptional Hermite polynomials}

One possible application of Theorem \ref{thm:detequiv} and the minimal
order problem is to provide a more efficient computation for
exceptional Hermite polynomials.  Exceptional Hermite polynomials are
complete families of orthogonal polynomials that arise as
eigenfunctions of a Sturm-Liouville problem on the real line,
\cite{Gomez-Ullate2009a,Gomez-Ullate2010c}. The degree sequence of
each family does not range over all positive integers, i.e. there is a
finite number of gaps or missing degrees.  Let
$\lambda=(\lambda_1,\dots,\lambda_\ell)$ be a partition and
$M\subset \Z$ the corresponding standard Maya diagram with
$M_+=\{m_1,\dots,m_\ell\}$ its positive elements as determined by
\eqref{eq:Mfromlambda}.  Following \cite{Gomez-Ullate2014,Gomez-Ullate2015b} we define
an infinite number of polynomials in the following manner:
\begin{equation}\label{eq:XHermitedef}
  H^{(\lambda)}_n=\Wr[H_{m_\ell},\dots,H_{m_1},H_{\ell-|\lambda|+n}],\quad
  n\notin M+|\lambda|-\ell
\end{equation}
By construction,
$H^{(\lambda)}_n$ is a polynomial with
\[\deg H^{(\lambda)}_n = \sum_{i=1}^\ell \left(m_i-i+1\right) +
\ell-|\lambda|+n - \ell = n\]
The degree sequence for the exceptional Hermite family indexed by
partition $\lambda$ is $\Z\setminus (M+|\lambda|-\ell)$. Thus, degrees
$0,1,\ldots ,|\lambda| -\ell-1$ and the degrees
$m_1+|\lambda|-\ell,\ldots, m_\ell + |\lambda|-\ell$ are missing, so
that the polynomial sequence $\{H^{(\lambda)}_n\}_n$ is missing a
total of $|\lambda|$ degrees, the codimension of the sequence.

Exceptional Hermite polynomials are a generalization the classical Hermite
family because they satisfy a Hermite-like differential equation
\begin{equation}
  \label{eq:XDE}
  T_\lambda \left[H^{(\lambda)}_n\right] = 2(N-n) H^{(\lambda)}_n,\qquad n\notin M+|\lambda|-\ell,
\end{equation}
where
\begin{equation}
  T_\lambda [y]:=y''-2\left(x+\frac{H_M'}{H_M} \right)
  y'+ \left(\frac{H_M''}{H_M}+ 2 x
    \frac{H_M'}{H_M} \right) y\label{eqn:hatT}.
\end{equation}

We say that $\lambda$ is an even partition if $\ell$ is even and
$\lambda_{2i-1} = \lambda_{2i}$ for every $i$.  If $\lambda$ is even,
then the exceptional Hermite polynomials $H^{(\lambda)}_n$ satisfy the
orthogonality relations
\[ \int_\Rset H^{(\lambda)}_n H^{(\lambda)}_m W_\lambda(x) dx =
\delta_{m,n} \sqrt{\pi}\, 2^{j+\ell} j!
\prod_{i=1}^{\ell}(j-m_i),\qquad n,m\notin M+|\lambda|-\ell
\]
where $j = n+\ell-N$, and the orthogonality weight is defined as
\[ W_\lambda(x) = \frac{\relax{e^{-x^2}}}{{H_M(x)^2}}.\]
Moreover, if $\lambda$ is an even partition then
$\text{span}\{H^{(\lambda)}_n:n\notin M+|\lambda|-\ell\}$ is dense in
$\mathrm{L}^ 2(\Rset,W_\lambda)$.

Exceptional polynomials appear in a number of applications in
mathematical physics, mostly as solutions to exactly solvable quantum
mechanical problems describing bound states
\cite{Gomez-Ullate2014,Odake2009a,Quesne2009}. They appear also in
connection with super-integrable systems
\cite{Post2012,Marquette2013a}, exact solutions to Dirac's equation
\cite{Schulze-Halberg2014}, diffusion equations and random processes
\cite{Ho2011b}, finite-gap potentials \cite{Hemery2010} and point
vortex models \cite{Kudryashov}.

From a mathematical point of view, the main results are concerned with the full classification of exceptional polynomials \cite{Gomez-Ullate2012a,Garcia-Ferrero2016}, properties of their zeros \cite{Gomez-Ullate2013,Horvath2015,Kuijlaars2015}, and recurrence relations \cite{Miki2015,durr,Odake2015,Gomez-Ullate2015b}.

Here we are concerned with the most economical presentation of a given
exceptional Hermite polynomial family. The definition
\eqref{eq:XHermitedef} involves the computation of a Wronskian
determinant of order $\ell+1$.  In light of the preceding results, this
order can be potentially reduced by replacing the Wronskian with an
appropriate pseudo-Wronskian.

Let $M\subset \Z$ be a Maya diagram, $m_1>m_2> \cdots$ the elements of
$M$ arranged in descending order, and $\lambda$ the corresponding
partition. Set
\[ O_r = \{ m_j +1 \colon \lambda_{j+1} + j = r,\; 0\leq j \leq \ell
\}\qquad r\geq 0.\]
Let $k_r = \max O_r$, or $-\infty$ if the latter is the empty set.
Visually, the elements of $O_r$ are the labels of the inside corners
that lie on the anti-diagonal $i+j=r$, with $k_r$ the largest such.  By
Corollary \ref{cor:MGO}, if $r$ is the minimal girth, then $k_r$ is
the largest minimal girth origin.
\begin{prop}\label{prop:shortest}
  Let $M\subset \Z$ be a Maya diagram, $m\notin M$ and
  $M'=M\cup \{ m \}$.  Let $r$ be the minimal girth of $M$. Let $r'$
  denote the minimal girth of $M'$ and $O'_{r'}$ the minimal girth origins of $M'$.
  Then one of the following mutually exclusive possibilities holds.
  \begin{enumerate}
  \item[(a)] If $m< k_r$ then $r' = r-1$ and
    $O'_{r'} = \{ k\in O_r \colon k>m \}$.
  \item[(b)] If $m=k_r$ then $r'=r$ and
    $O'_{r'} = \{ m+ 1\} \cup \{ k\in O_{r+1} \colon k>m \}$.
  \item[(c)] If $k_r<m<k_{r+1}$, then $r' =r$ and
    $O'_{r'} = \{ k\in O_{r+1} \colon k>m \}$.
  \item[(d)] If $m\geq \max \{k_r,k_{r+1}\}$ then $r'=r+1$ and
    $O'_{r'} = O_r\cup \{ k\in O_{r+2} \colon k>m \}$.
  \end{enumerate}
\end{prop}
\begin{proof}
  Let $B=\{ (i_k,j_k) \}_{k\in\Z}$ and $B'=\{ (i'_k,j'_k)\}_{k\in \Z}$
  be the bent diagrams for $M$ and $M'$, respectively. The proof is
  based on the following key observation:
  \[ (i'_k,j'_k) =
  \begin{cases}
    (i_k, j_k+1) & \text{ if } k\leq m \\
    (i_k-1,j_k) & \text{ if } k> m.
  \end{cases}\]
  We argue each of the above cases in turn.  Suppose that (a) holds.
  If $k\in O_r$ and $k>m$, then $i'_k+j'_k=r-1$. For all other $k$, we
  have $i'_k+j_k \geq r$.

  Suppose that (b) holds. Hence $(i_{m+1},j_{m+1}) = (i_m+1,j_m)$ and
  hence $i'_{m+1}+j'_{m+1} = i_m+j_m=r$.  If $k\leq m$, then
  $i'_k+j'_k \geq r+1$.   If $k>m$, then 
  $i'_k + j'_k = i_k  + j_k-1$ is equal to $r$ if $k\in O_{r+1}$ and
  is $>r$, otherwise.

  Suppose that (c) holds.  If $k\leq m$, then $i'_k+j'_k \geq r+1$.
  If $k>m$ then $i'_k+j'_k = r$ if $k\in O_{r+1}$ and is $>r$
  otherwise.

  Suppose that (d) holds.  If $k\leq m$ then $i'_k+j'_k= r+1$ if $k\in
  O_r$ and is $>r$ otherwise.  If $k>m$ then $i'_k+j'_k =r+1$ if $k\in
  O_{r+2}$ and is $>r$ otherwise.

\end{proof}


Let $M\subset \Z$ be a Maya diagram and $\lambda$ the corresponding
partition.  For $n\notin M+|\lambda|-\ell$ set
$M^{(\lambda)}_n = M\cup \{ n+\ell-|\lambda|\}$.  Observe that
$H^{(\lambda)}_n = (-1)^\ell H_{M^{(\lambda)}_n}$.  We therefore seek
the origin of minimal girth for $M^{(\lambda)}_n$.  Let
$r^{(\lambda)}_n = \min \{ |M^{(\lambda)}_n-k|\colon k\in \Z\}$ be the
minimal order for $H^{(\lambda)}_n$. We call any such $k$ that realizes this
minimum an \emph{origin of minimal order}.

\begin{cor}
  \label{cor:shortest}
  Let $r=r_M$ and $n\notin M+|\lambda|-\ell$. If $k_{r+1}>k_r$, then
  \begin{equation}
    \label{eq:shortest1}
    r^{(\lambda)}_n =
    \begin{cases}
      r-1 & \text{ if } n<k_r+|\lambda|-\ell ,\\
      r & \text{ if } k_r+|\lambda|-\ell \leq n < k_{r+1}+|\lambda|-\ell \\
      r+1 & \text{ if } n> k_{r+1}+|\lambda|-\ell
    \end{cases}
  \end{equation}
  An origin of minimal order for each of the above cases is,
  respectively, $k_r, k_{r+1}, k_r$.

  If $k_{r+1}<k_r$, then
  \begin{equation}
    \label{eq:shortest2}
    r^{(\lambda)}_n =
    \begin{cases}
      r-1 & \text{ if } n<k_r+|\lambda|-\ell ,\\
      r & \text{ if }  n =k_r+|\lambda|-\ell \\
      r+1 & \text{ if } n> k_r+|\lambda|-\ell
    \end{cases}
  \end{equation}
  An origin of minimal order for each of the above case is,
  respectively, $k_r, k_r+1, k_r$.
\end{cor}

\begin{example}
  Consider the case of $\lambda=(2,2,1,1)$ The corresponding Frobenius
  symbol is $(\emptyset | 1,2,4,5)$.  The minimal girth is $r=2$
  with $k_2=6$ the unique minimal girth origin.  The Frobenius symbol of $M-6$ is
  $(5,2\mid \emptyset)$.  Hence, pseudo-Wronskian of smallest order is
  \[ H_{M-6} = \begin{vmatrix}
    \th_{5} & \th_{6} \\
    \th_{2} & \th_{3}
  \end{vmatrix}\]
  with 
  \[ H_M=\Wr[H_1,H_2,H_4,H_5] = -2^5\times 24  H_{M-6}. \]

  The exceptional Hermite polynomials for this partition are given by
  \begin{equation} \label{ex:H} H^{(\lambda)}_n=
    \Wr[H_1,H_2,H_4,H_5,H_{n-2}],\qquad n\in \{2,5,8,9,10,\dots\}.
  \end{equation}
  We can replace the above $5\times 5$ determinant with a
  pseudo-Wronskian of smaller order.  Here $k_3 = 3<6$ and hence we
  must apply \eqref{eq:shortest2}. An origin of minimal order is $7$
  if $n=7$, and $6$ otherwise.  Explicitly,
  \begin{align*}
    H^{(\lambda)}_2 &= 2^{12}\times 120 \; \th_2 \\
    H^{(\lambda)}_5 &= 2^{12} \times 72\;\th_5 \\
    H^{(\lambda)}_8 &=     K_8 \begin{vmatrix}
      \th_6 & \th_7\\
      \th_3 & \th_4
    \end{vmatrix} \\
    H^{(\lambda)}_n &= 
    K_n
    \begin{vmatrix}
      \th_{5} & \th_{6} & \th_{7}\\
      \th_{2} & \th_{3} & \th_{4}\\
      \h_{n-8} & \h_{n-8}' & \h_{n-8}''
    \end{vmatrix},\qquad  n>8,
  \end{align*}
  where
  \[ K_n = -2^9 \times 24\times (n-3)(n-4)(n-6)(n-7) \]

\end{example}

\begin{example}
  Consider the case of $\lambda=(4,4,1,1)$ The corresponding Frobenius
  symbol is $(\emptyset \mid 7,6,2,1)$.  The minimal girth is $r=3$
  with $k_3=3$ the unique minimal girth origin.  The Frobenius symbol of $M-3$ is
  $(2\mid 4,3)$.  Hence, pseudo-Wronskian of smallest order is
  \[ H_{M-3} = \begin{vmatrix}
    \th_{2} & \th_{3}  & \th_4 \\
    H_{3} & H'_{3} & H''_3\\
    H_{4} & H'_{4} & H''_4
  \end{vmatrix}\]
  with 
  \[ H_M=\Wr[H_1,H_2,H_6,H_7] = 2^5\times 600\,  H_{M-3}. \]

  The exceptional Hermite polynomials for this partition are given by
  \begin{equation} \label{ex:H} H^{(\lambda)}_n=
    \Wr[H_1,H_2,H_6,H_7,H_{n-6}],\qquad n\in \{6,9,10,11,14,15,\dots\}.
  \end{equation}
  Here $k_4 = 8>3$ and hence we must apply \eqref{eq:shortest1}.  An
  origin of minimal order is $8$ if $9\leq n\leq 14$ and $3$
  otherwise.  Explicitly,
  \begin{align*}
    H^{(\lambda)}_6 &= 2^{14}\times 9\times 7\times 25 \; \Wr[H_3,H_4] \\
    H^{(\lambda)}_9 &= 2^{11} \times9\times 5 
                      \begin{vmatrix}
                        \th_2 & \th_3 & \th_4\\
                        \th_{3} & \th_{4} & \th_{5}\\
                        \th_7 & \th_8 & \th_9
                      \end{vmatrix},\\
    H^{(\lambda)}_{10} &= 2^{11} \times9\times 5 
                      \begin{vmatrix}
                        \th_2 & \th_3 & \th_4\\
                        \th_4 & \th_5 & \th_6\\
                        \th_7 & \th_8 & \th_9
                      \end{vmatrix},\\
    H^{(\lambda)}_{11} &= -2^{11} \times 3\times 25
                      \begin{vmatrix}
                        \th_3 & \th_4 & \th_5\\
                        \th_4 & \th_5 & \th_6\\
                        \th_7 & \th_8 & \th_9
                      \end{vmatrix},\\
    H^{(\lambda)}_n &= 
    K_n
    \begin{vmatrix}
      \th_{2} & \th_{3} & \th_{4} & \th_5\\
      \h_2 & \h_2' & \h''_2 & \h'''_2\\
      \h_4 & \h_4' & \h''_4 & \h'''_4\\
      \h_{n-9} & \h_{n-9}' & \h_{n-9}'' & \h'''_{n-9}
    \end{vmatrix},\qquad  n\geq 14
  \end{align*}
  where
  \[ K_n = -2^{10}  \times 45\times (n-7)(n-8)\]

\end{example}

\section{Application to rational solutions of Painleve IV}
In this section we apply our minimal order results to the description
of rational solutions of \PIV, the fourth Painlev\'e equation:
\[ y'' = \frac{(y')^2}{2 y} + \frac32 y^3 + 4t y^2+ 2(t^2-a)y +
\frac{b}{y},\quad y = y(t).\]
In order to connect \PIV\ to pseudo-Wronskians we need to recall the
notion of a factorization chain, and to describe the class of rational
solvable extensions of the harmonic oscillator.

A factorization chain is a sequence of Schr\"odinger operators
\[ L_i = -D_x^2 + U_i ,\quad D_x= \frac{d}{dx},\; U_i=U_i(x)\]
that is related by Darboux transformations,
\begin{equation}
  \label{eq:Dxform}
  \begin{aligned}
    L_i &= (D_x + f_i)(-D_x + f_i)+\lambda_i, \quad f_i = f_i(x),\\
    L_{i+1} &= (-D_x + f_i)(D_x + f_i)+\lambda_i.
  \end{aligned}
\end{equation}
It follows that $f_i$ is the solution of the Riccati equations
\[ f_i' + f_i^2  = U_i - \lambda_i,\quad -f_i' +f_i^2 = U_{i+1}- \lambda_i.\]
Equivalently, 
\begin{equation}
  \label{eq:Lipsii}
  L_i\psi_i = \lambda_i\psi_i,\qquad\text{where } f_i = \frac{\psi_i'}{\psi_i}.
\end{equation}
It also follows that the potentials of the chain are related by
\[ U_{i+n} =U_i - 2 \left( f'_i+ \cdots + f'_{i+n-1}\right).\]
Eliminating the potentials, we obtain a chain of coupled equations
\[ (f_i + f_{i+1})' + f_{i+1}^2 - f_i^2 = \alpha_i,\quad \alpha_i =
\lambda_{i} - \lambda_{i+1}.\]
We speak of an $n$-step cyclic factorization chain if $U_{n+1} = U_1+\Delta$
for some $n\in \N$ and constant 
\[ \Delta= -(\alpha_1 + \cdots + \alpha_n).\]

It is well known \cite{Adlerchains,veselov93} that  \PIV\ is
equivalent to the 3-step cyclic factorization chain
\begin{equation}
  \label{eq:3chain}
  \begin{aligned}
    (f_1 + f_2)' + f_2^2 - f_1 ^2 &= \alpha_1,\quad f_i= f_i(z),\; i=1,2,3\\
    (f_2 + f_3)' + f_3^2 - f_2 ^2 &=\alpha_2 \\
    (f_3 + f_1)' + f_1^2 - f_3 ^2 &=\alpha_3
  \end{aligned}
\end{equation}
The reduction of \eqref{eq:3chain} to \PIV\ is via the following
relations
\begin{gather}
  \label{eq:f123W}
  f_1 = -W - \frac{\Delta}{2} z,\quad f_{2,3} = \frac{W}{2} \pm \frac{
    W'-\alpha_2}{2W},\quad
  \Delta = -(\alpha_1+\alpha_2+\alpha_3),\; W=W(z),\\ \nonumber
  W(z) = 2^{-\frac12} \Delta^{\frac12}\, y(t),\; z=2^{\frac12}
  \Delta^{-\frac12}\, t,\qquad a = \Delta^{-1}(\alpha_3-\alpha_1),\; b
  = -2\Delta^{-2}\alpha_2^2.
\end{gather}

Next, we recall some relevant definitions from
\cite{Gomez-Ullate2014}.  A rational extension of the harmonic
oscillator is a potential of the form
\[ U(x) = x^2 + \frac{a(x)}{b(x)},\]
where $a(x), b(x)$ are polynomials with $\deg a< \deg b$.   We say
that the corresponding Schr\"odinger operator
$L = -D^2 + U$ is exactly solvable by polynomials if there exists
functions $\mu(x),\zeta(x)$ such that for all but finitely many $k\in \Nz$
there exists a degree $n$ polynomial $y_n(z)$ such that
\[ \psi_n  = \mu(x) y_n(\zeta(x)) \] is a (formal) eigenfunction of $L$ --- that  is 
\[ -\psi_n'' + U \psi_n = \lambda_n \psi_n,\]
for some constant $\lambda_n$.  
The following result was proved in
\cite{Gomez-Ullate2014}.
\begin{thm}
  Every rational extension of the harmonic oscillator that is exactly solvable by polynomials has the form
  \[ U(x) = x^2 - 2 D_x^2 \log \Wr[H_{m_1},\ldots, H_{m_\ell}]+C,\]
  where $\{m_1,\dots,m_\ell\}$ are positive integers and $C$ is an
  additive constant.
\end{thm}
\noindent
For a Maya diagram $M\subset \Z$, define
\begin{equation}
  \label{eq:UMdef}
  U_M = x^2 - 2 D_x^2 \log H_M + 2 |M_+|- 2 |M_-| .
\end{equation}
\begin{prop}
  Let $M\subset \Z$  and $M'=M+k,\; k\in\Z$ be equivalent  Maya  diagrams. Then
  \begin{equation}
    \label{eq:UM+k}
    U_{M'} = U_M+2k,\quad k\in \Z.
  \end{equation}
\end{prop}
\begin{proof}
  By Theorem \ref{thm:detequiv}, 
  \[ D_x^2 \log H_M = D_x^2 \log H_{M'}\]
  To establish the value of the shift in \eqref{eq:UM+k}, it suffices
  to consider the case of $k=1$.  The general case then follows by
  induction.  Now there are two subcases.  If $-1\in M$ then
  $M'_- = M_-$ and $M'_+ = M_++1$.  If $-1\notin M$, then
  $M'_- = M_--1$ and $M'_+ = M_+$.  Therefore, in both cases,
  \[ M'_+ - M'_- = M_+ - M_- + 1.\]
\end{proof}
Thus, 
the set of rational extensions of the harmonic oscillator modulo additive constants is in bijective
correspondence with the set of unlabelled Maya diagrams.  

\begin{definition}
  We define a Maya diagram chain to be a sequence of Maya diagrams
  $M_1, M_2,\ldots, M_\ell$ such that there exist
  $m_1,\ldots, m_\ell\in \Z$ satisfying
  \[ M_{i+1} =
  \begin{cases}
    M_{i} \cup \{ m_i \} & \text{ if } m_i \notin M_i,\\
    M_{i} \setminus \{ m_i \} & \text{ if } m_i \in M_i,
  \end{cases},\quad i=1,\ldots, \ell-1.
  \]
\end{definition}

The following are proved in \cite{GGUM1}.  Also see \cite{clarkson1}
and the references therein.
\begin{prop}
  \label{prop:isotonicmaya}
  Let $M_i,\; i=1,\ldots, \ell$ be a Maya diagram chain. Then,
  $L_i = -D^2 + U_{M_i},\; i=1,\ldots, \ell$ is a factorization chain.
  Conversely, every factorization chain of rational extensions of the harmonic oscillator that are solvable by polynomials arises in precisely
  this fashion.
\end{prop}
\begin{thm}
  Every rational solution of \PIV corresponds to a 3-cyclic chain of
 rational extensions of the harmonic oscillator.
\end{thm}


In light of the above, we wish to classify 3-cyclic chains of Maya
diagrams $M_1, M_2, M_3, M_4$ such that $M_4 = M_1+k$ for some
$k\in \Z$.  To accomplish this, we define the following classes of
Maya diagrams:
\begin{gather}
  \MGH(m,\ell) =  \Z_{-} \cup \{ j\in \Z \colon m\leq j
                 \leq m+\ell-1 \},\quad m,\ell \in \Nz,\\
  \MO(\ell_1,\ell_2) = \Z_{-} \cup \{ 3j+1 \colon
                0\leq j < \ell_1 \} \cup \{ 3j+2 \colon 0 \leq j
                < \ell_2 \},\quad \ell_1, \ell_2 \in \Nz,
\end{gather}
where $\Z_-$ is the set of negative integers.
For example,
\[
\MGH(2,5)_+ = \{ 2,3,4,5,6 \},\qquad \MO(2,5)_+ = \{
1,2,4,5,8,11,14\}.\]
We will call the former a Maya diagram of GH-type (Generalized
Hermite), and the latter, a Maya diagram of O-type (Okamoto).
Before proceeding, we note the following degeneracies of this
notation:
\[ \MGH(m,0) = \MO(0,0) = \Z_-,\quad \MGH(0,\ell) = \Z_- + \ell.\]

The following results are proved in \cite{GGUM1}.
\begin{prop}
  \label{prop:MGHU}
  Suppose that $M_i\subset \Z,\; i=1,\ldots, 4$ is a chain of Maya
  diagrams such that $M_4=M_1+k$ for some $k\in \Z$.  Then, up to
  translation, $M_1$ is either a Maya diagram of GH-type or of O-type.
  In the first case,
  \begin{equation}
    \label{eq:GHmaya}
    M_1 = \MGH(m,\ell),\quad M_4 = (M_1\setminus m) \cup \{ 0, m+\ell
    \},\quad m,\ell \in \Nz .
  \end{equation}
  In the second case,
  \begin{equation}
    \label{eq:Umaya}
    M_1 = \MO(\ell_1,\ell_2),\quad M_4 = M_1 \cup \{ 0, 3\ell_1+1,
    3\ell_2+2 \},\quad \ell_1,\ell_2\in \Nz.
  \end{equation}
\end{prop}
\begin{prop}
  The rational solutions of \PIV\ that arise from 3-cyclic chains of
  the harmonic oscillator fall into one of two classes.  The rational
  solutions of GH-type take the form:
  \begin{align}
    \label{eq:GH1}
    y &= D_t \left[\log
        \frac{H_{\MGH(m,\ell)}}{H_{\MGH(m,\ell+1)}}\right]_{x=t},
    && a =-(1+m+2\ell),\; &&b = -2 m^2,\\
    \label{eq:GH2}
    y &= D_t \left[\log
        \frac{H_{\MGH(m,\ell)}}{H_{\MGH(m-1,\ell)}}\right]_{x=t},
    &&   a=2m+\ell-1,\; &&b=-2\ell^2,\; m>0,\\
    \label{eq:GH3}
    y &= D_t \left[\log
        \frac{H_{\MGH(m,\ell)}}{H_{\MGH(m+1,\ell-1)}}\right]_{x=t}-2t,
    && a=\ell-m-1,\; && b = -2(m+\ell)^2,\; \ell>0.
  \end{align}
  The rational solutions of O-type take the form:
  \begin{align}
    \label{eq:U1}
    y &= -\frac23t+D_t \left[\log
        \frac{H_{\MO(\ell_1,\ell_2)}}{H_{\MO(\ell_1-1,\ell_2-1)}}\right]_{x=
        \frac{t}{\sqrt{3}}}
        ,&&   a= \ell_1+\ell_2,\; && b = -\frac29
           (1-3\ell_1+3\ell_2)^2  \\
    \label{eq:U2}
    y &=  -\frac23t+D_t\left[ \log
        \frac{ H_{\MO(\ell_1,\ell_2)}}{ H_{\MO(\ell_1+1,\ell_2)}}
        \right]_{x=\frac{t}{\sqrt{3}}},
         &&   a= -1-2\ell_1+\ell_2,\; && b = -\frac29 (2+3\ell_2)^2  \\
    \label{eq:U3}
    y &=  -\frac23t+D_t\left[ \log
        \frac{ H_{\MO(\ell_1,\ell_2)}}{ H_{\MO(\ell_1,\ell_2+1)}}
        \right]_{x=\frac{t}{\sqrt{3}}},
         &&   a= -2-2\ell_2+\ell_1,\;&&  b = -\frac29
                    (1+3\ell_1)^2  
  \end{align}
\end{prop}

Since the rational solutions of \PIV\ are linear combinations of
log-derivatives of Hermite Wronskians, it makes sense to apply the
shortest chain results of Section \ref{sec:min} to obtain more
economical descriptions of these rational solutions.
\begin{prop}\label{prop:GH}
  If $m,\ell >0$, then the minimal order of a GH-type Maya diagram
  $M=\MGH(m,\ell)$ is $\min\{ m,\ell\}$.  If $\ell\leq m$, then the
  minimal order determinant is just the usual Wronskian
  \[ H_M = \Wr[H_{m},\ldots, H_{m+\ell-1}].\]
  If $\ell > m$ then the minimal-order determinant is the pseudo-Wronskian
  \[ H_{M-m-\ell} = \begin{vmatrix}
    \th_{\ell}&\ldots& \th_{\ell+m-1} \\
    \vdots & \ddots & \vdots \\
    \th_{\ell+m-1}&  \ldots & \th_{\ell+2m-2}
    \end{vmatrix}
    \]
\end{prop}
\begin{example}
 For the rational solutions of \PIV corresponding to the generalized Hermite class, take $m=2,\ell=4$  and set 
  \begin{align*}
     M_0&=\MGH(2,4)= \Z_-\cup \{ 2,3,4,5\}\\
    M_1 &= \MGH(2,5) = \Z_- \cup \{ 2,3,4,5,6\}\\
    M_2 &= \MGH(1,4) = \Z_- \cup \{ 1,2,3,4\}\\
    M_3 &= \MGH(3,3) = \Z_- \cup \{ 3,4,5 \}.    
  \end{align*}
  The corresponding Wronskians are determinants of order $4,5,4,3$,
  respectively. The more economical description is in term of
  pseudo-Wronskian determinants of order $2,2,1,3$, respectively.  By
  a direct calculation,
  \begin{align*}
    H_{M_0-6} &= \begin{vmatrix}
    \th_5 & \th_6\\
    \th_4 & \th_5
  \end{vmatrix} = -32(16x^8+64x^6+120x^4+45)\\
    H_{M_1-7} &= \begin{vmatrix}
      \th_6 & \th_7\\
      \th_5 & \th_6
  \end{vmatrix} = -64(32 x^{10}+240 x^8 + 720 x^6 + 600 x^4 + 450 x^2
            - 225)\\
    H_{M_2-5} &= \th_5 = 4 (4x^4+12x^2+3)\\
    H_{M_3-6} &=  \begin{vmatrix}
      \th_5 & \th_6 & \th_7\\
      \th_4 & \th_5 & \th_6\\
      \th_3 & \th_4 & \th_5
    \end{vmatrix} = -512 x ( 16 x^8+72 x^4 - 135)
  \end{align*}
  Setting $x=t$, the corresponding rational solutions \eqref{eq:GH1}
  \eqref{eq:GH2} \eqref{eq:GH3} are
  \begin{align*}
    y_1 &= \frac{32 \left(4 t^7+12 t^5+15 t^3\right)}{16 t^8+64
          t^6+120 t^4+45}-\frac{20 \left(16 t^9+96 t^7+216 t^5+120
          t^3+45 t\right)}{32 t^{10}+240 
          t^8+720 t^6+600 t^4+450 t^2-225} ,&&\quad (a,b) = (-11,-8),\\
    y_2 &=- \frac{32 \left(4 t^7+12 t^5+15 t^3\right)}{16 t^8+64
          t^6+120 t^4+45}-\frac{8 \left(2 t^3+3 t\right)}{4 t^4+12
          t^2+3} ,&&\quad (a,b) = (7,-32),\\
    y_3&=-\frac{1}{t}-\frac{32 \left(4 t^7+9 t^3\right)}{16 t^8+72
         t^4-135}+\frac{32 \left(4 t^7+12 t^5+15 t^3\right)}{16 t^8+64
         t^6+120 t^4+45}-2 t,&&\quad (a,b) = (1,-72).
  \end{align*}
\end{example}

\begin{prop}\label{prop:O}
  The minimal order of an O-type Maya diagram
  $M=\MO(\ell_1,\ell_2),\; \ell_1,\ell_2\in \N$, is
  $\max \{ \ell_1,\ell_2\}$.  If $\ell_1\leq \ell_2$, then the
  minimal-order pseudo-Wronskian is
  \[ H_{M-3\ell_1} = \pm
  \begin{vmatrix}
    \th_{2} & \th_{3} & \ldots & \th_{\ell_2+1} \\
    \th_{5} & \th_{6} & \ldots & \th_{\ell_2+4}   \\
    \vdots & \vdots & \ddots & \vdots \\    
    \th_{3\ell_1-1} &   \th_{3\ell_1}  & \ldots &   \th_{3\ell_1+\ell_2-2}\\
    H_2 & H_2' & \ldots & H_2^{(\ell_2-1)} \\
    H_5 & H_5' & \ldots & H_5^{(\ell_2-1)} \\
    \vdots & \vdots & \ddots & \vdots \\    
    H_{3\ell_2-3\ell_1-1} & H_{3\ell_2-3\ell_1-1}' & \ldots &
    H_{3\ell_2-3\ell_1-1}^{(\ell_2-1)}  
  \end{vmatrix}
  \]
  If $\ell_1> \ell_2$, then the minimal-order pseudo-Wronskian is
  \[ H_{M-3\ell_2} = \pm
  \begin{vmatrix} \th_{2} & \th_{3} & \ldots & \th_{\ell_1+1} \\
    \th_{5} & \th_{6} & \ldots & \th_{\ell_1+4} \\ \vdots & \vdots &
    \ddots & \vdots \\ \th_{3\ell_2-1} & \th_{3\ell_2} & \ldots &
    \th_{3\ell_2+\ell_1-2}\\ 
    H_1 & H_1' & \ldots & H_1^{(\ell_2-1)} \\ 
    H_4    & H_4' & \ldots & H_4^{(\ell_2-1)} \\ 
    \vdots & \vdots & \ddots &   \vdots \\ 
    H_{3\ell_1-3\ell_2-2} & H_{3\ell_1-3\ell_2-2}' & \ldots &
    H_{3\ell_1-3\ell_2-2}^{(\ell_2-1)}
  \end{vmatrix}
\]
\end{prop}
\begin{proof}
  Suppose that $\ell_1\leq \ell_2$, so that
  \[ M_+ = \{ 1,2, 4,5, \ldots, 3\ell_1-2, 3\ell_1-1\} \cup \{
  3\ell_1+2, 3\ell_1+5,\ldots, 3\ell_2-1 \}.\]The
  inside corners are located at $0,3,6,\ldots, 3\ell_1$ and then
  $3\ell_1+3,3 \ell_1+6,\ldots, 3\ell_2$.  The corresponding girths
  initially decrease:
  $\ell_1+\ell_2,+\ell_2-1, \ldots, \ell_2+1, \ell_2$ and increase by
  $+1$ thereafter.

  Suppose that $\ell_1> \ell_2$, so that
  \[ M_+ = \{ 1,2, 4,5, \ldots, 3\ell_2-2, 3\ell_2-1\} \cup \{
  3\ell_2+1, 3\ell_2+5,\ldots, 3\ell_1-2 \}.\]The
  inside corners are $0,3,6,\ldots, 3\ell_2$ and then
  $3\ell_2+2,3 \ell_2+5,\ldots, 3\ell_1-1$.  The corresponding girths
  initially decrease:
  $\ell_1+\ell_2,\ell_1+\ell_2-1, \ldots, \ell_1+1, \ell_1$ and
  increase by $+1$ thereafter.
\end{proof}

Figure \ref{fig:okamoto} contains a graphical illustration for the
minimal order pseudo-Wronskians described by Propositions
\ref{prop:GH} and \ref{prop:O} for the Generalized Hermite and Okamoto
polynomials, respectively. We see thus that in the case of generalized
Hermite polynomials, the minimal order occurs always for the partition
itself or its conjugate partition, and therefore it is always a true
Wronskian. In the example of the figure for $GH(3,5)$, the minimal
order according to Proposition \ref{prop:GH} must be $\min(3,5)=3$:
\[ GH(3,5)=\Wr(H_3,H_4,H_5,H_6,H_7)=\frac{1}{18432}\Wr(\th_5,\th_6,\th_7)\]
On the contrary, for Okamoto polynomials, the minimal order is generically a pseudo-Wronskian. In the example of the figure, for $O(3,5)$ that has order 8, the equivalent minimal order pseudo-Wronskian has order $\max(3,5)=5$ and corresponds to the Durfee symbol $[6,4,2|4,2]_{3\times2}$.
\[ O(3,5)= \Wr(H_1,H_2,H_4,H_5,H_7,H_8,H_{11},H_{14} ) \propto {\rm e}^{-3 x^2} \Wr ({\rm e}^{x^2} \th_8, {\rm e}^{x^2} \th_5, {\rm e}^{x^2} \th_2, H_2,H_5 )  \]

\begin{figure}
  \begin{center}
  \includegraphics[width=0.6\textwidth]{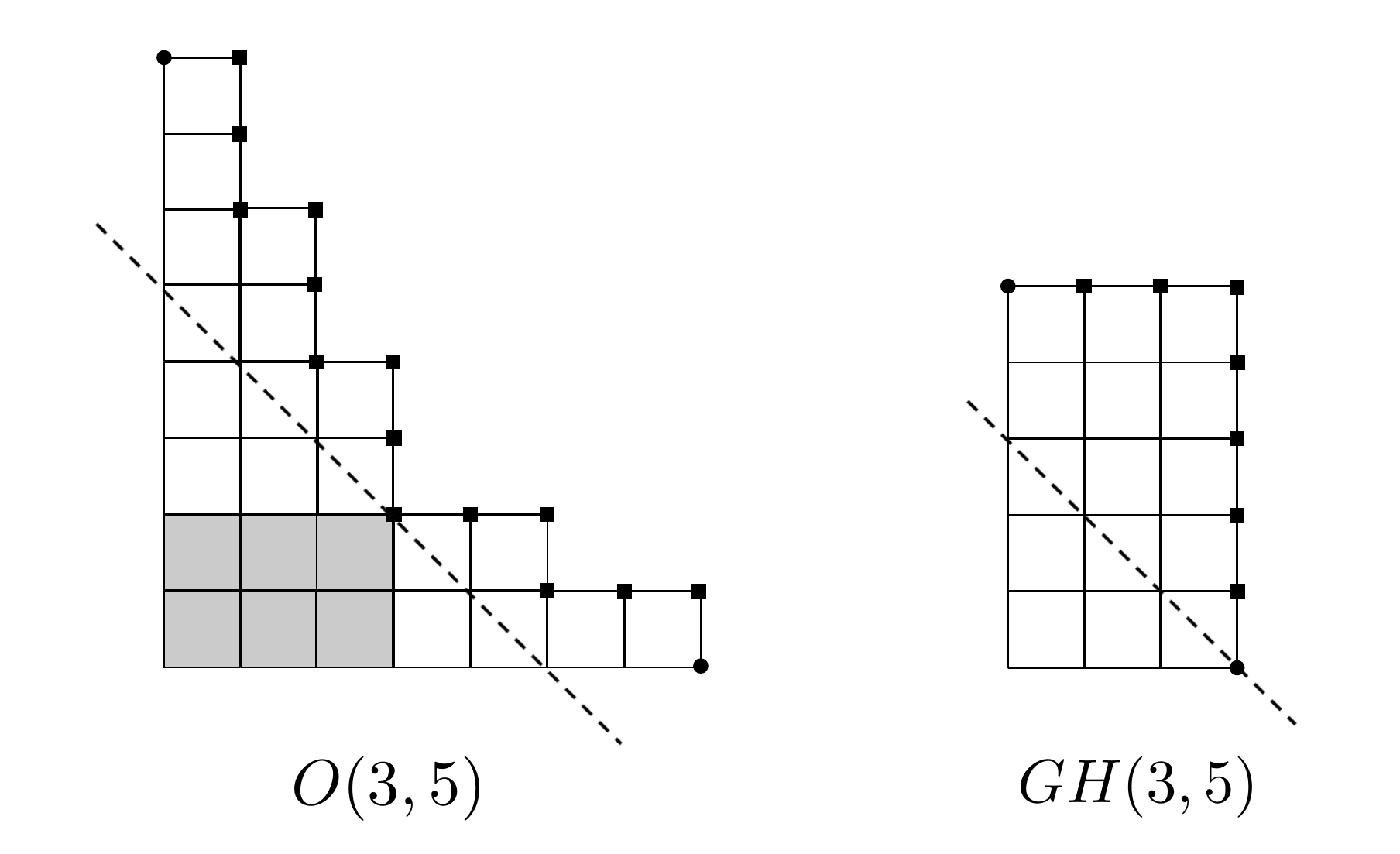}
\caption{Partitions and minimal order pseudo-Wronskians for two generalized Hermite and Okamoto polynomials. Minimal order representations for these polynomials are given by Durfee symbols $[6,4,2|4,2]_{3\times2}$ for $O(3,5)$ and  $[5,5,5|\emptyset]_{3\times 0}$ for $GH(3,5)$. }
  \label{fig:okamoto}
  \end{center}
\end{figure}

\begin{example}
  For the rational solutions of \PIV corresponding to the Okamoto class, take $\ell_1=1,\ell_2=2$ and set 
  \begin{align*}
     M_0 &= \MO(1,2) =  \Z_- \cup \{ 1,2,5\} \\
     M_1 &= \MO(0,1) = \Z_- \cup \{ 2\} \\
    M_2  &= \MO(2,2) = \Z_- \cup \{ 1,2,4,5\}\\
    M_3  &= \MO(1,3) = \Z_- \cup \{ 1,2,5,8\}.
  \end{align*}
  The corresponding Wronskian determinants have order $3,1,4,4$,
  respectively.  The minimal order description is in term of
  pseudo-Wronskian determinants of order $2,1,2,3$, respectively.  By
  a direct calculation,
  \begin{align*} 
    H_{M_0-3} &= 
    \begin{vmatrix} \th_2 & \th_2 \\      H_2 & H_2'
    \end{vmatrix} = -8x (4x^4-5)\\
    H_{M_1} &= \h_2 = 2(2x^2-1)\\ 
    H_{M_2-6} &= 
    \begin{vmatrix} \th_5 & \th_6 \\ \th_2 & \th_3 
    \end{vmatrix} = -48(8x^6+20x^4+10x^2+5)\\
    H_{M_3-3} &= 
    \begin{vmatrix} \th_2 & \th_3  & \th_4 \\ H_2 & H_2' & H_2''\\ H_5 & H_5' &
      H_5''
    \end{vmatrix} = 192(32x^{10}- 80x^8 - 80 x^6 + 200 x^4 - 150 x^2 - 25)
  \end{align*} Setting $x=t/\sqrt{3}$, the corresponding rational
  solutions \eqref{eq:U1}   \eqref{eq:U2} \eqref{eq:U3} are 
\begin{align*} y_1 
  &   = -\frac{2 t}{3}+\frac{16 t^3}{4
    t^4-45}+\frac{1}{t}-\frac{4 t}{2 t^2-3} ,
  & & \quad (a,b) = \left(3,-\frac{32}{9}\right),\\
  y_2 
  &=-\frac{2 t}{3}+\frac{16 t^3}{4 t^4-45}+\frac{1}{t}-\frac{12
        \left(4 t^5+20 t^3+15 t\right)}{8 t^6+60 t^4+90 t^2+135},
  &&\quad (a,b) =  \left(-1,-\frac{128}{9}\right),\\
  y_3&=-\frac{2 t}{3}+\frac{16 t^3}{4 t^4-45}+\frac{1}{t}-\frac{20
       \left(16 t^9-96 t^7-216 t^5+1080 t^3-1215 t\right)}{32
       t^{10}-240 t^8-720 t^6+5400 t^4-12150 t^2-6075},
  &&\quad (a,b) =  \left(-5,-\frac{32}{9}\right).
  \end{align*}
\end{example}

\section{Acknowledgements}

The research of DGU has been supported in part by Spanish MINECO-FEDER
Grant MTM2015-65888-C4-3,  the ICMAT-Severo Ochoa project
SEV-2015-0554 and the BBVA Foundation \textit{Grant for researchers and cultural creators}.  The research of RM was supported in part by NSERC grant
RGPIN-228057-2009. DGU and RM would like to thank Universit\'e de
Lorraine for their hospitality during their visit in the summer of
2015 where many of the results in this paper where first obtained.

\end{document}